\documentclass[11pt]{article}
\usepackage{amsfonts,amssymb,amsmath,amsthm}

\usepackage{bm}
\usepackage{palatino}
\usepackage{mathpazo}
\usepackage{inconsolata}
\usepackage[hidelinks]{hyperref}

\usepackage{xcolor}
\usepackage{varwidth}
\usepackage{multicol}

\begingroup
    \makeatletter
    \@for\theoremstyle:=definition,remark,plain\do{%
        \expandafter\g@addto@macro\csname th@\theoremstyle\endcsname{%
            \addtolength\thm@preskip\parskip
            }%
        }
\endgroup

\usepackage[margin=1in]{geometry}

\theoremstyle{definition}
\newtheorem{theorem}{Theorem}

\newtheorem{conjecture}[theorem]{Conjecture}

\newtheorem{question}{Question}

\numberwithin{theorem}{section}
\numberwithin{question}{section}
\numberwithin{definition}{section}
\numberwithin{example}{section}
\numberwithin{remark}{section}

\newcommand*\patchAmsMathEnvironmentForLineno[1]{%
  \expandafter\let\csname old#1\expandafter\endcsname\csname #1\endcsname
  \expandafter\let\csname oldend#1\expandafter\endcsname\csname end#1\endcsname
  \renewenvironment{#1}%
     {\linenomath\csname old#1\endcsname}%
     {\csname oldend#1\endcsname\endlinenomath}}%
\newcommand*\patchBothAmsMathEnvironmentsForLineno[1]{%
  \patchAmsMathEnvironmentForLineno{#1}%
  \patchAmsMathEnvironmentForLineno{#1*}}%
\AtBeginDocument{%
\patchBothAmsMathEnvironmentsForLineno{equation}%
\patchBothAmsMathEnvironmentsForLineno{align}%
\patchBothAmsMathEnvironmentsForLineno{flalign}%
\patchBothAmsMathEnvironmentsForLineno{alignat}%
\patchBothAmsMathEnvironmentsForLineno{gather}%
\patchBothAmsMathEnvironmentsForLineno{multline}%
}

\usepackage{parskip}
\usepackage{lineno}
\usepackage[small]{titlesec}

\usepackage[square,comma,numbers,sort&compress]{natbib}

\usepackage{color}
\definecolor{todocolor}{RGB}{205,235,139}
\definecolor{todo-idea}{RGB}{120,180,255}
\definecolor{todo-error}{RGB}{208,31,60}
\definecolor{todo-question}{RGB}{255,255,136}

\usepackage[colorinlistoftodos, color=todocolor, textsize=small]{todonotes}

\usepackage{colortbl}

\usepackage{tabularray}

\newcommand{\Av}{\operatorname{Av}}

\newcommand{\CC}{\mathcal{C}}
\newcommand{\DD}{\mathcal{D}}
\newcommand{\JJ}{\mathcal{J}}
\newcommand{\GG}{\mathcal{G}}

\DeclareMathOperator{\perm}{perm}

\newcommand{\gk}[1]{GK.{#1}}
\newcommand{\pplink}[1]{\footnote{\url{https://permpal.com/perms/id/#1/}}}

\newcommand{\gkoneposet}{\
\begin{tikzpicture}[scale=0.4, baseline=(current bounding box.center)]
	\draw[fill] (0,0) circle (6pt) node[below=1pt] {$2$};
	\draw[fill] (0.7,0) circle (6pt) node[below=1pt] {$3$};
	\draw[fill] (1.4, 0) circle (6pt) node[below=1pt] {$4$};
	\draw[fill] (0.7, 0.8) circle (6pt) node[shift={(-6pt,3pt)}] {$1$};
	\draw[fill] (0.7,1.6) circle (6pt) node[shift={(-6pt,4pt)}] {$5$};
	\draw[ultra thick] (0, 0) -- (0.7, 0.8);
	\draw[ultra thick] (1.4, 0) -- (0.7, 0.8);
	\draw[ultra thick] (0.7, 0.8) -- (0.7,1.6);
\end{tikzpicture}\
}

\newcommand{\broom}[5]{\
\begin{tikzpicture}[scale=0.4, baseline=(current bounding box.center)]]
	\draw[fill] (0,0) circle (6pt) node[below=1pt] {$#1$};
	\draw[fill] (0.7,0) circle (6pt) node[below=1pt] {$#2$};
	\draw[fill] (1.4, 0) circle (6pt) node[below=1pt] {$#3$};
	\draw[fill] (0.7, 0.8) circle (6pt) node[shift={(-6pt,3pt)}] {$#4$};
	\draw[fill] (0.7,1.6) circle (6pt) node[shift={(-6pt,4pt)}] {$#5$};
	\draw[ultra thick] (0, 0) -- (0.7, 0.8);
	\draw[ultra thick] (0.7, 0) -- (0.7, 0.8);
	\draw[ultra thick] (1.4, 0) -- (0.7, 0.8);
	\draw[ultra thick] (0.7, 0.8) -- (0.7,1.6);
\end{tikzpicture}\
}

\newcommand{\gktwoposet}{\broom{2}{3}{4}{1}{5}}

\newcommand{\widediamond}[5]{\
\begin{tikzpicture}[scale=0.4, baseline=(current bounding box.center)]
	\draw[fill] (1.6, 0)   circle (6pt) node[shift={(-6pt, -4pt)}] {$#1$};
	\draw[fill] (0.0, 0.9) circle (6pt) node[left] {$#2$};
	\draw[fill] (1.6, 0.9) circle (6pt) node[left] {$#3$};
	\draw[fill] (3.2, 0.9) circle (6pt) node[right] {$#4$};
	\draw[fill] (1.6, 1.8) circle (6pt) node[shift={(-6pt, 4pt)}] {$#5$};
	\draw[ultra thick] (1.6, 0)   -- (0,   0.9);
	\draw[ultra thick] (1.6, 0)   -- (1.6, 0.9);
	\draw[ultra thick] (1.6, 0)   -- (3.2, 0.9);
	\draw[ultra thick] (0  , 0.9) -- (1.6, 1.8);
	\draw[ultra thick] (1.6, 0.9) -- (1.6, 1.8);
	\draw[ultra thick] (3.2, 0.9) -- (1.6, 1.8);
\end{tikzpicture}\
}

\newcommand{\treepop}[5]{\
    \begin{tikzpicture}[scale=0.4, baseline=(current bounding box.center)]
        \draw[fill] (0, -1.6) circle (6pt) node[left] {$#1$};
        \draw[fill] (1.4,-1.6) circle (6pt) node[right] {$#2$};
        \draw[fill] (-0.7, -0.8) circle (6pt) node[left] {$#3$};
        \draw[fill] (0.7, -0.8) circle (6pt) node[shift={(7pt, 4pt)}] {$#4$};
        \draw[ultra thick] (0, 0) -- (1.4, -1.6);
        \draw[ultra thick] (0, 0) -- (-0.7, -0.8);
        \draw[ultra thick] (0.7, -0.8) -- (0, -1.6);
        \draw[fill] (0,0) circle (6pt) node[shift={(7pt, 4pt)}] {$#5$};
    \end{tikzpicture}\
}
\newcommand{\treepopshifted}[5]{\
    \begin{tikzpicture}[scale=0.4, baseline={([yshift=-0.5ex]current bounding box.center)}]
        \draw[fill] (0, -1.6) circle (6pt) node[left] {$#1$};
        \draw[fill] (1.4,-1.6) circle (6pt) node[right] {$#2$};
        \draw[fill] (-0.7, -0.8) circle (6pt) node[left] {$#3$};
        \draw[fill] (0.7, -0.8) circle (6pt) node[shift={(7pt, 4pt)}] {$#4$};
        \draw[ultra thick] (0, 0) -- (1.4, -1.6);
        \draw[ultra thick] (0, 0) -- (-0.7, -0.8);
        \draw[ultra thick] (0.7, -0.8) -- (0, -1.6);
        \draw[fill] (0,0) circle (6pt) node[shift={(7pt, 4pt)}] {$#5$};
    \end{tikzpicture}\
}

\newcommand{\gkthreeposet}{\treepopshifted{4}{5}{2}{3}{1}}

\newcommand{\fishpop}[5]{\
	\begin{tikzpicture}[scale=0.4, baseline=(current bounding box.center)]
        \draw[fill] (0,0) circle (6pt) node[shift={(-6pt, -3pt)}] {$#1$};
        \draw[fill] (-0.7, 0.8) circle (6pt) node[left] {$#2$};
        \draw[fill] (0, 0.8) circle (6pt) node[shift={(7pt, 0pt)}] {$#3$};
        \draw[fill] (2, 0.8) circle (6pt) node[right] {$#4$};
        \draw[fill] (0, 1.6) circle (6pt) node[shift={(-6pt,3pt)}] {$#5$};
        \draw[ultra thick] (0, 0) -- (0, 1.6);
        \draw[ultra thick] (0, 0) -- (-0.7, 0.8);
        \draw[ultra thick] (0, 0) -- (2, 0.8);
        \draw[ultra thick] (2, 0.8) -- (0, 1.6);
    \end{tikzpicture}\
}

\newcommand{\xpop}[5]{\
    \begin{tikzpicture}[scale=0.4, baseline=(current bounding box.center)]
        \draw[fill] (0,0) circle (6pt) node[left] {$#1$};
        \draw[fill] (1.6, 0) circle (6pt) node[right] {$#2$};
        \draw[fill] (0.8, 1) circle (6pt) node[left=1pt] {$#3$};
        \draw[fill] (0, 2) circle (6pt) node[left] {$#4$};
        \draw[fill] (1.6,2) circle (6pt) node[right] {$#5$};
        \draw[ultra thick] (0, 0) -- (1.6, 2);
        \draw[ultra thick] (1.6, 0) -- (0, 2);
    \end{tikzpicture}\
}

\newcommand{\xpopshifted}[5]{\
    \begin{tikzpicture}[scale=0.4, baseline={([yshift=-0.5ex]current bounding box.center)}]
        \draw[fill] (0,0) circle (6pt) node[left] {$#1$};
        \draw[fill] (1.6, 0) circle (6pt) node[right] {$#2$};
        \draw[fill] (0.8, 1) circle (6pt) node[left=1pt] {$#3$};
        \draw[fill] (0, 2) circle (6pt) node[left] {$#4$};
        \draw[fill] (1.6,2) circle (6pt) node[right] {$#5$};
        \draw[ultra thick] (0, 0) -- (1.6, 2);
        \draw[ultra thick] (1.6, 0) -- (0, 2);
    \end{tikzpicture}\
}

\newcommand{\gkfourposet}{\xpopshifted{3}{4}{5}{1}{2}}

\newcommand{\housepop}[5]{\
    \begin{tikzpicture}[scale=0.4, baseline=(current bounding box.center)]
        \draw[fill] (0,0) circle (6pt) node[left] {$#1$};
        \draw[fill] (1.6, 0) circle (6pt) node[right] {$#2$};
        \draw[fill] (0, 1) circle (6pt) node[left] {$#3$};
        \draw[fill] (1.6, 1) circle (6pt) node[right] {$#4$};
        \draw[fill] (0.8,2) circle (6pt) node[right=1pt] {$#5$};
        \draw[ultra thick] (0, 0) -- (0, 1) -- (0.8, 2);
        \draw[ultra thick] (1.6, 0) -- (1.6, 1) -- (0.8, 2);
        \draw[ultra thick] (0, 0) -- (1.6, 1);
        \draw[ultra thick] (1.6, 0) -- (0, 1);
    \end{tikzpicture}\
}

\newcommand{\flagpopshifted}[5]{\
    \begin{tikzpicture}[scale=0.4, baseline={([yshift=-0.5ex]current bounding box.center)}]
        \draw[fill] (0,0) circle (6pt) node[left] {$#1$};
        \draw[fill] (0, 1) circle (6pt) node[left] {$#2$};
        \draw[fill] (1, 1) circle (6pt) node[right] {$#3$};
        \draw[fill] (0, 2) circle (6pt) node[left] {$#4$};
        \draw[fill] (1,2) circle (6pt) node[right] {$#5$};
        \draw[ultra thick] (0, 0) -- (0, 2);
        \draw[ultra thick] (1, 1) -- (1, 2);
        \draw[ultra thick] (0, 1) -- (1, 2);
        \draw[ultra thick] (1, 1) -- (0, 2);
    \end{tikzpicture}\
}

\newcommand{\gkfiveposet}{\flagpopshifted{3}{5}{4}{2}{1}}

\newcommand{\gksixposet}{\
    \begin{tikzpicture}[scale=0.4, baseline={([yshift=-0.5ex]current bounding box.center)}]
        \draw[fill] (0,0) circle (6pt) node[left=2pt] {$3$};
        \draw[fill] (1, 0) circle (6pt) node[right=2pt] {$4$};
        \draw[fill] (0.5, 0.8) circle (6pt) node[right=2pt] {$5$};
        \draw[fill] (0.5, 1.6) circle (6pt) node[right=2pt] {$1$};
        \draw[fill] (0.5, 2.4) circle (6pt) node[right=2pt] {$2$};
        \draw[ultra thick] (0, 0) -- (0.5, 0.8);
        \draw[ultra thick] (1, 0) -- (0.5, 0.8);
        \draw[ultra thick] (0.5, 0.8) -- (0.5, 2.4);
    \end{tikzpicture}\
}

\newcommand{\mclass}[6]{%
\begin{tikzpicture}[scale=0.4, baseline=(current bounding box.center)]
    \draw[fill] (0,0) circle (6pt) node[left] {\scriptsize $#1$};
    \draw[fill] (0.6,0.8) circle (6pt) node[left] {\scriptsize $#2$};
    \draw[fill] (1.2,0) circle (6pt) node[left] {\scriptsize $#3$};
    \draw[fill] (1.8,0.8) circle (6pt) node[left] {\scriptsize $#4$};
    \draw[fill] (2.4,0) circle (6pt) node[left] {\scriptsize $#5$};
    \draw[ultra thick] (0, 0) -- (0.6, 0.8) -- (1.2, 0) -- (1.8,0.8) -- (2.4,0);
    \node at (1.2, -1) {\small \texttt{#6}};
\end{tikzpicture}}

\newcommand{\sixclass}[6]{%
\begin{tikzpicture}[scale=0.4, baseline=(current bounding box.center)]
	\draw[fill] (0,0) circle (6pt) node[left] {\scriptsize $#1$};
	\draw[fill] (0.6,0.8) circle (6pt) node[left] {\scriptsize $#2$};
	\draw[fill] (1.2,0) circle (6pt) node[left] {\scriptsize $#3$};
	\draw[fill] (1.8,0.8) circle (6pt) node[left] {\scriptsize $#4$};
	\draw[fill] (2.4,0) circle (6pt) node[left] {\scriptsize $#5$};
    \draw[fill] (3,0.8) circle (6pt) node[left] {\scriptsize $#6$};
	\draw[ultra thick] (0, 0) -- (0.6, 0.8) -- (1.2, 0) -- (1.8,0.8) -- (2.4,0) -- (3,0.8);
\end{tikzpicture}}

\renewenvironment{abstract}{
	\begin{list}{}%
	{\setlength{\rightmargin}{1in}%
	\setlength{\leftmargin}{1in}}%
	\item[]\ignorespaces\begin{small}}%
	{\end{small}\unskip\end{list}%
}

\newpagestyle{main}[\small]{
	\headrule
	\sethead[\usepage][][]
	{\sc }{}{\usepage}
}

\title{Permutations avoiding bipartite partially ordered patterns have a regular insertion encoding}

\usepackage{makecell}
\author{
	\begin{tabular}{m{2.5in}m{2.5in}}
		\makecell{
			Christian Bean\\
			\small School of Computer Science\\
			\small and Mathematics\\
			\small Keele University\\
			\small Keele, United Kingdom\\
			\small \texttt{c.n.bean@keele.ac.uk}
		}&
		\makecell{
			\'Emile Nadeau\\
			\small Department of Computer Science\\
			\small Reykjavik University\\
			\small Reykjavik, Iceland\\
		}\\&\\
		\makecell{
			Jay Pantone\\
			\small Department of Mathematical\\
			\small and Statistical Sciences\\
			\small Marquette University\\
			\small Milwaukee, WI, USA\\
			\small \texttt{jay.pantone@marquette.edu}
		}
		&
		\makecell{
			Henning Ulfarsson\\
			\small Department of Computer Science\\
			\small Reykjavik University\\
			\small Reykjavik, Iceland\\
			\small \texttt{henningu@ru.is}
		}
	\end{tabular}
}

\date{}

\usepackage{arydshln}
\usepackage{makecell}

\begin{document}
\maketitle

\begin{abstract}
	We prove that any class of permutations defined by avoiding a partially ordered pattern (POP) with height at most two has a regular insertion encoding and thus has a rational generating function. Then, we use Combinatorial Exploration to find combinatorial specifications and generating functions for hundreds of other permutation classes defined by avoiding a size $5$ POP, allowing us to resolve several conjectures of Gao and Kitaev and of Chen and Lin.
\end{abstract}


\section{Introduction}

A \emph{permutation} of size $n$ is a linear ordering of the elements $\{1, \ldots, n\}$. When we want to refer to individual entries in a permutation $\pi$ of size $n$, we write $\pi = \pi(1)\pi(2)\cdots\pi(n)$. We say that a permutation $\pi$ of size $n$ \emph{contains} a permutation $\sigma$ of size $k \leq n$ if $\pi$ contains a (not necessarily consecutive) subword $\pi(i_1)\pi(i_2)\cdots\pi(i_k)$ with $i_1 < i_2 < \cdots < i_k$ such that $\pi(i_\ell) < \pi(i_m)$ if and only if $\sigma(\ell) < \sigma(m)$; in other words, this subword is order-isomorphic to $\sigma$. In this context, we often refer to $\sigma$ as a \emph{pattern}. For example, the permutation $3714652$ of size $7$ contains the pattern $132$ because the subword $\pi(3)\pi(6)\pi(7) = 152$ is order-isomorphic to $132$. There are eight other occurrences of $132$ in $\pi$ as well. On the other hand, $\pi$ \emph{avoids} (i.e., does not contain) the pattern $3124$.

This containment relation on permutations defines an infinite poset on the set of all permutations of all sizes. Sets of permutations that are downward closed in this poset are called \emph{permutation classes}. Each permutation class can be uniquely described by the (possibly infinite) minimal set of permutations that it does not contain, called its \emph{basis}. We use the notation $\Av(B)$ to denote the class with basis $B$. Permutation classes have been extensively studied; we refer the interested reader to Kitaev's book on the subject~\cite{kitaev:pp-book} and Vatter's comprehensive survey~\cite{vatter:perm-survey}.

A \emph{partially ordered pattern (POP)} of size $k$ is a poset on $k$ elements labeled with the symbols $\{1, \ldots, k\}$. We say that a permutation $\pi$ of size $n$ contains a POP $P$ of size $k$ if $\pi$ contains a (not necessarily consecutive) subword $\pi(i_1)\pi(i_2)\cdots\pi(i_k)$ with $i_1 < i_2 < \cdots < i_k$ such that $\pi(i_\ell) < \pi(i_m)$ if $\ell < m$ in $P$. Compare this to the definition of pattern avoidance above; for two elements $\ell$ and $m$ that are incomparable in $P$, there is no restriction on the relative order of the values $\pi(i_\ell)$ and $\pi(i_m)$ in an occurrence of the pattern $P$ in $\pi$.

Consider for example the poset $P = \begin{tikzpicture}[scale=0.4, baseline=(current bounding box.center)]
		\draw[fill] (0,0) circle (6pt) node[left=4pt] {$3$};
		\draw[fill] (0, 0.8) circle (6pt) node[left=4pt] {$1$};
		\draw[fill] (-0.7, 1.6) circle (6pt) node[left=4pt] {$2$};
		\draw[fill] (0.7,1.6) circle (6pt) node[right=4pt] {$4$};
		\draw[ultra thick] (0, 0) -- (0, 0.8) -- (-0.7,1.6);
		\draw[ultra thick] (0, 0.8) -- (0.7,1.6);
	\end{tikzpicture}$. The permutation $\pi = 4726135$ contains the subword $\pi(3)\pi(4)\pi(5)\pi(7) = 2615$ at indices $(i_1, i_2, i_3, i_4) = (3,4,5,7)$ which satisfies the condition above. For instance, $3 < 1$ in $P$ and $\pi(i_3) < \pi(i_1)$. Since $2$ and $4$ are incomparable in $P$, we do not require either $\pi(i_2) < \pi(i_4)$ or $\pi(i_2) > \pi(i_4)$. In fact, swapping $\pi(i_2)$ and $\pi(i_4)$ gives a new permutation $\pi' = 4725136$ that also contains an occurrence $P$ at the same indices.

If $\sigma$ is a permutation that contains a POP $P$ and $\pi$ is a permutation that contains $\sigma$, then $\pi$ also contains $P$. This implies that the set of permutations that avoid a given POP is downward closed in the permutation containment order and therefore is a permutation class. We use the notation $\Av(P)$ to denote the class of permutations avoiding the POP $P$. It will always be clear in context whether we are using this $\Av(P)$ notation or the earlier $\Av(B)$ notation for the permutations avoiding the basis $B$.

It turns out that it is easy to compute the basis of $\Av(P)$, as Gao and Kitaev~\cite{gao:partially-ordered-patterns} implicitly observe. A \emph{linear extension} of a poset is a total order on its elements that obeys the relations of the poset.
\begin{theorem}
	\label{theorem:basis}
	For a POP $P$, the basis of the permutation class $\Av(P)$ is the set of permutations that are the (group-theoretic) inverse of a linear extension of $P$.
\end{theorem}
Before proving Theorem~\ref{theorem:basis}, we present an example demonstrating why inverses of permutations need to be considered. Let $P = \begin{tikzpicture}[scale=0.4, baseline=(current bounding box.center)]
		\draw[fill] (0,0) circle (6pt) node[left=3pt] {\scriptsize $2$};
		\draw[fill] (0, 0.8) circle (6pt) node[left=3pt] {\scriptsize $3$};
		\draw[fill] (0, 1.6) circle (6pt) node[left=3pt] {\scriptsize $1$};
		\draw[fill] (0, 2.4) circle (6pt) node[left=3pt] {\scriptsize $4$};
		\draw[ultra thick] (0, 0) -- (0, 2.4);
	\end{tikzpicture}$. As $P$ is already a total order, its only linear extension is itself. If a permutation $\pi$ contains $P$ at the indices $i_1 < i_2 < i_3 < i_4$, the relations $2 < 3 < 1 < 4$ in $P$ imply that $\pi(i_2) < \pi(i_3) < \pi(i_1) < \pi(i_4)$. Therefore $\pi$ contains the pattern $2314^{-1} = 3124$ at these indices. This shows that $\Av(P) = \Av(3124)$.

\begin{proof}
	Let $P$ be a POP and let $\{L_1, L_2, \ldots, L_N\}$ be the set of linear extensions of $P$. We first claim that a permutation avoids $P$ if and only if it avoids each $L_i$. To see this, first assume that the permutation $\pi$ contains an occurrence of $P$ at the subword $\pi(i_1)\pi(i_2)\cdots\pi(i_k)$. Reorder these entries so that $\pi(i_{\ell_1}) < \pi(i_{\ell_2}) < \cdots < \pi(i_{\ell_k})$. This shows that the subword $\pi(i_1)\pi(i_2)\cdots\pi(i_k)$ is an occurrence of total order $T = \ell_1 < \ell_2 < \cdots < \ell_k$. If $a < b$ in $P$, then $\pi(i_a) < \pi(i_b)$ which implies that $a < b$ in $T$. Therefore $T$ is a linear extension of $P$. Conversely, assume $\pi$ contains some linear extension $T$ of $P$. An occurrence of $T$ in $\pi$ is also an occurrence of $P$ because the relations of $P$ are a subset of the relations of $T$. This completes our first claim, and we can now conclude that the set of permutations avoiding $P$ is equal to the set of permutations that avoid all linear extensions of $P$.

	For a total order $L = \ell_1 < \ell_2 < \cdots < \ell_k$, we define $\perm(L)$ to be the permutation $(\ell_1\ell_2\cdots\ell_k)^{-1}$. We now claim that $\pi$ contains the total order $T = \tau(1) < \tau(2) < \cdots < \tau(k)$ if and only if $\pi$ contains the pattern $\perm(T)$. To this end, assume $\pi(i_1)\pi(i_2)\cdots \pi(i_k)$ is an occurrence of $T$, so that $\pi(i_{\tau(1)}) < \pi(i_{\tau(2)}) < \cdots < \pi(i_{\tau(k)})$. Let $\sigma$ denote the pattern contained in $\pi$ at these same indices. Since the smallest entry of this occurrence is $\pi(i_{\tau(1)})$ it follows that $\sigma(\tau(1)) = 1$. Similarly, $\sigma(\tau(j)) = j$ for all $1 \leq j \leq k$. Therefore $\sigma = \tau^{-1} = \perm(T)$. Conversely, suppose $\pi(i_1)\pi(i_2)\cdots\pi(i_k)$ is an occurrence of $\sigma$. Of the entries in the subword, the one with the smallest value is $\pi(i_{\sigma^{-1}(1)})$, while the next smallest value is $\pi(i_{\sigma^{-1}(2)})$ and so on. Therefore, $\pi(i_{\sigma^{-1}(1)}) < \pi(i_{\sigma^{-1}(2)}) < \cdots < \pi(i_{\sigma^{-1}(k)})$ and therefore the subword is an occurrence of the total order $\sigma^{-1}(1) < \sigma^{-1}(2) < \cdots < \sigma^{-1}(k)$.

	From the two preceding claims, it follows that permutations avoiding $P$ are precisely those avoiding the inverse of all linear extensions of $P$, and therefore the basis of $\Av(P)$ is $\{\perm(L_1),\allowbreak \perm(L_2),\allowbreak \cdots,\allowbreak \perm(L_N)\}$.
\end{proof}

We call a permutation class $\CC$ a \emph{POP class} if $\CC = \Av(P)$ for some POP $P$. Not all permutation classes are POP classes, e.g., there is no poset $P$ such that $\Av(P) = \Av(1234, 1432)$. Theorem~\ref{theorem:basis} gives an easy test for whether a class is a POP class. Given any permutation class $\Av(B)$, form the total orders labeled by the inverses of the permutations in $B$. Let $P$ be the intersection of these total orders (i.e., the intersection of the sets of relations each implies). If $\Av(B)$ is a POP class, then $P$ must be its associated poset, and so Theorem~\ref{theorem:basis} can be used to check whether the basis of $\Av(P)$ is $B$. The Sage~\cite{sage} script in the Github repository \url{https://github.com/jaypantone/bipartite-pops} can be used to check if a permutation class is a POP class.

The number of permutation classes avoiding a subset of patterns of size $n$ is, of course, $2^{n!}$, but the number of POP classes of size $n$ is given by the sequence $1, 1, 3, 19, 219, 4231, 130023, \ldots$ (starting with $n=0$), which is A001035 in the OEIS~\cite{oeis}. Section~\ref{section:symmetries} explores the topic of the number of POP classes taking certain symmetries into account.

In 2005, Kitaev~\cite{kitaev:OG-pop} introduced the notion of \emph{partially ordered patterns} (POPs) as a way of studying generalized permutations and other objects. Later, in 2019, Gao and Kitaev~\cite{gao:partially-ordered-patterns} studied the version of POP classes that we defined above, in particular determining the enumeration of several infinite families of POP classes and many specific POP classes whose posets have size $4$. They included several tables of conjectural links between POP classes whose posets have size $4$ or $5$ and other combinatorial objects.  Several of these conjectures were later proved by Yap, Wehlau, and Zaguia~\cite{yap:pops}. Recently, Kitaev and Pyatkin~\cite{kitaev:pop2}, among other things, enumerated many POP classes whose poset has height (length of the longest chain) at most $2$. They call such posets \emph{bipartite} and so we will use the same terminology, but this should not be confused with the notion of a bipartite graph. Every bipartite poset has a Hasse diagram whose graph is bipartite, but not every poset with a bipartite Hasse graph has height at most $2$.

In Section~\ref{section:rie}, we show that the generating function of any POP class whose poset is bipartite can be algorithmically computed using techniques from the field of permutation patterns, proving the following theorem.
\begin{theorem}
	\label{theorem:rie}
	A POP class has a regular insertion encoding if and only if it is bipartite.
\end{theorem}

In Section~\ref{section:symmetries}, we discuss symmetries of the POP containment relation, and we use Theorem~\ref{theorem:rie} to compute the enumeration of many new bipartite POP classes, extending the classifications given by Kitaev and Pyatkin. In Section~\ref{section:tilescope}, we use the Combinatorial Exploration framework of Albert, Bean, Claesson, Nadeau, Pantone, and Ulfarsson~\cite{combinatorial-exploration} to derive combinatorial specifications and compute generating functions for all POP classes of size $4$ and many non-bipartite POP classes whose posets have size $5$. Section~\ref{section:conjectures} addresses six conjectures of Gao and Kitaev~\cite{gao:partially-ordered-patterns}. Within Subsection~\ref{subsection:gk4}, we prove a conjecture of Chen and Lin~\cite{chen:length-5-patts-and-inv-201-210} regarding the enumeration of nine permutation classes.

\section{Bipartite POP Classes Have Regular Insertion Encodings}
\label{section:rie}

The insertion encoding, introduced by Albert, Linton, and Ru\v{s}kuc~\cite{albert:insertion-encoding}, is a language-theoretic approach that describes how a permutation is built by repeatedly adding a new maximum element. When this language is regular, Vatter~\cite{vatter:regular-insertion-encoding} gave an algorithm for computing the rational generating function. A simple test determines whether a permutation class has a regular insertion encoding. In order to describe the test, we need the following definition.

The \emph{vertical juxtaposition} of two classes $\CC$ and $\DD$ is the set of permutations for which there is some value $i$ such that the entries strictly below $i$ form a permutation in $\CC$ and the entries above and including $i$ form a permutation in $\DD$. For example, Figure~\ref{figure:juxta} shows the four juxtapositions relevant to this work. The vertical juxtaposition $\JJ_1$ contains the set of permutations consisting of an increasing sequence on top of an increasing sequence. The juxtaposition of two permutation classes is a permutation class, and Atkinson~\cite{atkinson:restricted-perms} gave a method for computing its basis. This method shows $\JJ_1 = \Av(321, 2143, 2413)$, $\JJ_2 = \Av(123, 3142, 3412)$, $\JJ_3 = Av(132, 312)$, and $\JJ_4 = \Av(213, 231)$.

\begin{figure}
	\begin{center}
		\begin{tikzpicture}
			\draw[ultra thick, lightgray] (0,0) grid (1,2);
			\draw[ultra thick, black] (0.15, 0.15) -- (0.85, 0.85);
			\draw[ultra thick, black] (0.15, 1.15) -- (0.85, 1.85);
			\node at (0.5, -0.35) {$\JJ_1$};
		\end{tikzpicture}
		\qquad
		\begin{tikzpicture}
			\draw[ultra thick, lightgray] (0,0) grid (1,2);
			\draw[ultra thick, black] (0.15, 0.85) -- (0.85, 0.15);
			\draw[ultra thick, black] (0.15, 1.85) -- (0.85, 1.15);
			\node at (0.5, -0.35) {$\JJ_2$};
		\end{tikzpicture}
		\qquad
		\begin{tikzpicture}
			\draw[ultra thick, lightgray] (0,0) grid (1,2);
			\draw[ultra thick, black] (0.15, 0.85) -- (0.85, 0.15);
			\draw[ultra thick, black] (0.15, 1.15) -- (0.85, 1.85);
			\node at (0.5, -0.35) {$\JJ_3$};
		\end{tikzpicture}
		\qquad
		\begin{tikzpicture}
			\draw[ultra thick, lightgray] (0,0) grid (1,2);
			\draw[ultra thick, black] (0.15, 0.15) -- (0.85, 0.85);
			\draw[ultra thick, black] (0.15, 1.85) -- (0.85, 1.15);
			\node at (0.5, -0.35) {$\JJ_4$};
		\end{tikzpicture}
	\end{center}
	\caption{The four vertical juxtapositions used to characterize permutation classes with a regular insertion encoding.}
	\label{figure:juxta}
\end{figure}
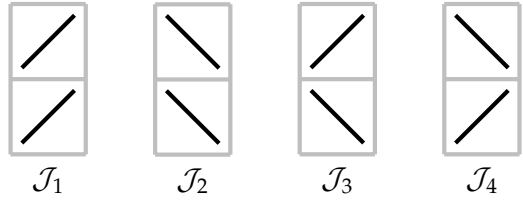
\begin{theorem}[{\cite[Proposition 13]{albert:insertion-encoding}}]
	\label{theorem:rie-condition}
	A class $\CC = \Av(B)$ has a regular insertion encoding if and only if each of the four permutation classes $\JJ_1$, $\JJ_2$, $\JJ_3$, and $\JJ_4$ contains at least one basis element $\beta$ from $B$.
\end{theorem}

In other words, the class $\CC = \Av(B)$ has a regular insertion encoding if among the elements of $B$, at least one is the juxtaposition of two increasing sequences, at least one is the juxtaposition of two decreasing sequences, at least one is the juxtaposition of an increasing sequence on top of a decreasing sequence, and at least one is the juxtaposition of a decreasing sequence on top of an increasing sequence.

We prove Theorem~\ref{theorem:rie} in two parts.
\begin{theorem}
	\label{theorem:rie-to-bip}
	If a POP class has a regular insertion encoding, its associated poset is bipartite.
\end{theorem}
\begin{proof}
	Suppose $\CC = \Av(B)$ is a POP class associated with a non-bipartite poset $P$, implying that $P$ contains a chain with three elements, $a < b < c$. Thus, for every element $\beta\in B$, we must have $\beta(a) < \beta(b) < \beta(c)$. This implies that the subword at indices $a$, $b$, and $c$ of each basis element is an occurrence of the same pattern.

	As there exists some size three permutation that is contained in every $\beta \in B$, and as all six size $3$ permutations appear in one of the bases of $\JJ_1$, $\JJ_2$, $\JJ_3$, and $\JJ_4$, it must be the case that one of these four classes contains no element of $B$. Therefore $\CC$ does not have a regular insertion encoding.
\end{proof}

For example, the class associated with the poset
\begin{tikzpicture}[scale=0.3, baseline={([yshift=-1ex]current bounding box.center)}]]
	\draw[fill] (0,0) circle (6pt) node[left] {\scriptsize $3$};
	\draw[fill] (0, 0.8) circle (6pt) node[left] {\scriptsize$1$};
	\draw[fill] (-0.7, 1.6) circle (6pt) node[above] {\scriptsize$2$};
	\draw[fill] (0, 1.6) circle (6pt) node[above] {\scriptsize$4$};
	\draw[fill] (0.7,1.6) circle (6pt) node[above] {\scriptsize$5$};
	\draw[ultra thick] (0, 0) -- (0, 1.6);
	\draw[ultra thick] (0, 0.7) -- (-0.7, 1.5);
	\draw[ultra thick] (0, 0.7) -- (0.7, 1.5);
\end{tikzpicture}
has a basis containing the six permutations of size $5$ that have the value $1$ at index $3$, have the value $2$ at index $1$, and the values $3$, $4$, and $5$ in any order at the indices $2$, $4$, $5$; the six permutations are $23145$, $23154$, $24135$, $24153$, $25134$, and $25143$. Now consider the chain $3 < 1 < 5$. The existence of this chain implies that the subword $\beta(1)\beta(3)\beta(5)$ is an occurrence of $213$ in all six basis elements. Since $213$ is a basis element of $\JJ_4$, it follows that $\JJ_4$ does not contain any of the six basis elements.

\begin{theorem}
	\label{theorem:bip-to-rie}
	A POP class whose associated poset is bipartite has a regular insertion encoding.
\end{theorem}
\begin{proof}
	Let $\CC = \Av(B)$ be a POP class with associated bipartite poset $P$ with $n$ elements. Let $L$ denote the elements of $P$ that are the lower elements in a chain of size two and let $U$ denote the remaining elements. Assume $|L| = k$. We will construct, for each class $\JJ_1$, $\JJ_2$, $\JJ_3$, and $\JJ_4$, a linear extension of $P$ whose corresponding permutation is in the class.

	Any permutation in which the entries at the indices in $L$ have the values $\{1, \ldots, k\}$ and the entries at the indices in $U$ have the values $\{k+1, \ldots, n\}$ is in $B$. (Note that $B$ contains other permutations as well.) The permutation in which the entries at indices in $L$ are increasing, and the entries at indices in $U$ are increasing is in $\JJ_1$. Similarly, by swapping one or both of the increasing sequences for a decreasing sequence, we find permutations in $B$ from each of $\JJ_2$, $\JJ_3$, and $\JJ_4$. Therefore, by Theorem~\ref{theorem:rie-condition}, $\CC$ has a regular insertion encoding.
\end{proof}

We illustrate this proof with an example as well. Let $\CC$ be the POP class with the associated poset
\begin{tikzpicture}[scale=0.3, baseline={([yshift=-0.5ex]current bounding box.center)}]]
	\draw[fill] (0,0) circle (6pt) node[below] {\scriptsize $2$};
	\draw[fill] (0.4, 1) circle (6pt) node[above] {\scriptsize $5$};
	\draw[fill] (0.8,0) circle (6pt) node[below] {\scriptsize $4$};
	\draw[fill] (1.2, 1) circle (6pt) node[above] {\scriptsize $1$};
	\draw[fill] (1.6,0) circle (6pt) node[below] {\scriptsize $6$};
	\draw[fill] (2.4, 0) circle (6pt) node[below] {\scriptsize $3$};
	\draw[ultra thick] (0,0) -- (0.4,1) -- (0.8,0)--(1.2,1)--(1.6,0);
\end{tikzpicture}.
In this case $L = \{2,4,6\}$ and $U = \{1, 3, 5\}$. We then construct the four permutations $415263 \in \JJ_1$, $615243 \in \JJ_2$, $435261 \in \JJ_3$,  and $635241 \in \JJ_4$. These are the four permutations with the entries $123$ or $321$ at the indices in $L$ and the entries $456$ or $654$ at the indices in $U$.

Theorems~\ref{theorem:rie-to-bip} and~\ref{theorem:bip-to-rie} together prove our main result, Theorem~\ref{theorem:rie}.

Vatter~\cite{vatter:regular-insertion-encoding} describes an algorithm to determine the generating function of any class with a regular insertion encoding. The Combinatorial Exploration framework \cite{combinatorial-exploration}, which we discuss in the next section, can perform this computation more efficiently. Consequently, we can use Theorem~\ref{theorem:rie} to rederive some of the enumerative results in \cite{gao:partially-ordered-patterns}, and \cite{kitaev:pop2} in a uniform manner and address several of their open questions.

In light of Theorem~\ref{theorem:rie} it is natural to wonder if perhaps tripartite POPs (those whose longest chain has length $3$) have some similarly nice characterization of their generating functions. This is unlikely, given the class $\Av(1234, 1324) = \Av\left(\begin{tikzpicture}[scale=0.3, baseline={([yshift=-0.5ex]current bounding box.center)}]
			\draw[fill] (0,0) circle (6pt) node[left] {\footnotesize $1$};
			\draw[fill] (-0.7, 0.8) circle (6pt) node[left] {\footnotesize $2$};
			\draw[fill] (0.7, 0.8) circle (6pt) node[right] {\footnotesize $3$};
			\draw[fill] (0,1.6) circle (6pt) node[above=1pt, right] {\footnotesize $4$};
			\draw[ultra thick] (0, 0) -- (-0.7, 0.8) -- (0,1.6);
			\draw[ultra thick] (0, 0) -- (0.7, 0.8) -- (0,1.6);
			\useasboundingbox ([xshift=0.5in,yshift=0.5in]current bounding box.south west) rectangle ([xshift=-0.5in,yshift=-0.5in]current bounding box.north east);
		\end{tikzpicture}\right)$ which is conjectured by Albert, Homberger, Pantone, Shar, and Vatter~\cite{albert:c-machines} to have a non-D-finite generating function.

\section{Equivalences of POP Classes}
\label{section:symmetries}

\subsection{Symmetric Equivalence}
\label{subsection:symmetry}

The \emph{reverse} of a permutation $\pi = \pi(1)\pi(2)\cdots\pi(n)$ is the permutation $\pi^R = \pi(n)\cdots\pi(2)\pi(1)$, the \emph{complement} of $\pi$ is the permutation $\pi^C$ defined by $\pi^C(i) = n+1-\pi(i)$, and we have already discussed the \emph{inverse} $\pi^{-1}$. These three operations preserve pattern containment, e.g., if $\sigma$ is contained in $\pi$, then $\sigma^R$ is contained in $\pi^R$. We extend these operations to permutation classes by defining, for example, $\CC^R = \{\pi^R : \pi \in \CC\}$. If $\CC = \Av(B)$, then $\CC^R = \Av(\{\beta^R : \beta \in B\})$, and similarly for complement and inverse.

If $\CC$ can be obtained from $\DD$ by some sequence of these three operations, then we say that $\CC$ is \emph{symmetrically equivalent} to $\DD$, and it's easy to see that in this case $\CC$ and $\DD$ have the same counting sequence. The resulting equivalence classes have sizes $1$, $2$, $4$, or $8$.

While there are $2^{n!}$ permutation classes avoiding a (possibly empty) subset of permutations of size $n$, the number of equivalence classes under symmetry is, of course, smaller, approximately one-eighth the size. Starting with $n=1$, the counting sequence starts $2, 3, 21, 2139264$, and the OEIS sequence A277086 gives a triangle of numbers that refines these numbers by the size of the basis; the numbers above is the sequence of row sums of that triangle. When restricting only to those permutation classes that are POP classes, the counting sequence for the number of equivalence classes under symmetry starting with $n=1$ is $1, 2, 7, 64, 1068, 32651$. We added this sequence to the OEIS as entry A366705, but have not found a formula to quickly compute further terms.

Gao and Kitaev~\cite{gao:partially-ordered-patterns} give two operations on a poset $P$ that produce a new poset $P'$ such that that $\Av(P)$ and $\Av(P')$ are symmetrically equivalent. First, complementing the labels of a poset $P$ of size $n$, by which we mean replacing each label $i$ with $n+1-i$, gives a poset $P'$ such that $\Av(P') = \Av(P)^R$. Second, forming $P'$ by reversing each relation of $P$ (i.e., if $a < b$ in $P$, then $b < a$ in $P'$), we have that $\Av(P') = \Av(P)^C$. These operations are insufficient to generate the complete equivalence classes under symmetry of each POP class. For example, the class $\CC = \Av(1342, 1432)$ is a POP class with poset $P = \begin{tikzpicture}[scale=0.3, baseline={([yshift=-0.5ex]current bounding box.center)}]
		\draw[fill] (0,0) circle (6pt) node[left] {\footnotesize $1$};
		\draw[fill] (0, 0.8) circle (6pt) node[left] {\footnotesize $4$};
		\draw[fill] (-0.7, 1.6) circle (6pt) node[left] {\footnotesize $2$};
		\draw[fill] (0.7,1.6) circle (6pt) node[right] {\footnotesize $3$};
		\draw[ultra thick] (0, 0) -- (0, 0.8);
		\draw[ultra thick] (0, 0.8) -- (-0.7, 1.6);
		\draw[ultra thick] (0, 0.8) -- (0.7, 1.6);
	\end{tikzpicture}$. All eight permutation classes in the symmetry class are POP classes. In particular, $\CC^{-1} = \Av(1423, 1432)$ corresponds to the poset $\begin{tikzpicture}[scale=0.3, baseline={([yshift=-0.5ex]current bounding box.center)}]
		\draw[fill] (0,0) circle (6pt) node[left] {\footnotesize $1$};
		\draw[fill] (-0.7, 0.8) circle (6pt) node[left] {\footnotesize $3$};
		\draw[fill] (0.7, 0.8) circle (6pt) node[right] {\footnotesize $4$};
		\draw[fill] (0,1.6) circle (6pt) node[above=1pt, right] {\footnotesize $2$};
		\draw[ultra thick] (0, 0) -- (-0.7, 0.8) -- (0,1.6);
		\draw[ultra thick] (0, 0) -- (0.7, 0.8) -- (0,1.6);
	\end{tikzpicture}$, which does not correspond to either of the two operations described by Gao and Kitaev. We therefore pose the following question.

\begin{question}
	When is the inverse of a POP class also a POP class? When it is, is there an operation on posets that can be defined in a uniform way to transform the poset of one class into the poset of its inverse, without an intermediate translation into a set of permutations.
\end{question}

This question is trivial for classes defined by avoiding a single permutation since these POP classes are defined by posets that are chains. However, $\Av(123, 132, 231)$ is a POP class, while its inverse $\Av(123, 132, 312)$ is not.

\subsection{Wilf Equivalence}

Two permutation classes are \emph{Wilf-equivalent} if they have the same number of permutations of each size. Clearly, symmetrically equivalent classes are Wilf-equivalent, but there are many examples of non-symmetrically equivalent classes that are also Wilf-equivalent, a classic example being $\Av(123)$ and $\Av(132)$. When we are interested in computing the generating function of many POP classes, we can therefore restrict ourselves to one POP class from each symmetry class.

Gao and Kitaev~\cite{gao:partially-ordered-patterns} noted several examples of Wilf-equivalence among non-symmetrically equivalent POP classes, and later Kitaev and Pyatkin~\cite{kitaev:pop2} provided a complete classification of the Wilf-equivalence classes of POP classes whose poset has the shape $\begin{tikzpicture}[scale=0.4, baseline={([yshift=-0.5ex]current bounding box.center)}]
		\draw[fill] (0,0) circle (6pt);
		\draw[fill] (0,0.8) circle (6pt);
		\draw[fill] (0.8,0) circle (6pt);
		\draw[fill] (0.8,0.8) circle (6pt);
		\draw[ultra thick] (0, 0) -- (0, 0.8) -- (0.8, 0) -- (0.8,0.8);
	\end{tikzpicture}$, which they call $N$-patterns. They later ask whether a similar classification can be performed for $M$-patterns, those with posets of the shape $\begin{tikzpicture}[scale=0.4, baseline={([yshift=-0.5ex]current bounding box.center)}]
		\draw[fill] (0,0) circle (6pt);
		\draw[fill] (0.4,0.8) circle (6pt);
		\draw[fill] (0.8,0) circle (6pt);
		\draw[fill] (1.2,0.8) circle (6pt);
		\draw[fill] (1.6,0) circle (6pt);
		\draw[ultra thick] (0, 0) -- (0.4, 0.8) -- (0.8, 0) -- (1.2,0.8) -- (1.6,0);
	\end{tikzpicture}$. As these posets are all bipartite, Theorem~\ref{theorem:rie} shows that their generating functions can be algorithmically computed, and this extends to what we will call \emph{zig-zag classes} of arbitrary size, those whose poset has shape $\begin{tikzpicture}[scale=0.4, baseline={([yshift=-0.5ex]current bounding box.center)}]
		\draw[fill] (0,0) circle (6pt);
		\draw[fill] (0.4,0.8) circle (6pt);
		\draw[fill] (0.8,0) circle (6pt);
		\draw[fill] (1.2,0.8) circle (6pt);
		\draw[fill] (1.6,0) circle (6pt);
		\draw[ultra thick] (0, 0) -- (0.4, 0.8) -- (0.8, 0) -- (1.2,0.8) -- (1.6,0) -- (1.8, 0.6);
		\node at (2.6, 0.6) {\scriptsize $\bm{\cdots}$};
	\end{tikzpicture}$. If $P$ is a zig-zag poset of size $n$, then the size of the basis of the class $\Av(P)$ is the $n$th Euler number, as proved by Stanley~\cite[Exercise 3.66(c)]{stanley:ec1}.  The number of zig-zag classes of size $n$ is $n!$ when $n$ is even and $n!/2$ when $n$ is odd.

We have performed the computation necessary to determine the Wilf-equivalence classes for zig-zag classes of size $5$ ($M$-classes) and of size $6$ using the Tilescope software package~\cite{tilings-bibtex} developed by Albert, Bean, Claesson, Nadeau, Pantone, and Ulfarsson~\cite{combinatorial-exploration}, which we describe in more detail in the next section. Figure~\ref{fig:M-classes} shows the classification of the $60$ $M$-classes into 23 Wilf-equivalence classes. The $720$ zig-zag classes of size $6$ split into $177$ Wilf-equivalence classes, of which only five are non-trivial. These five are shown in Figure~\ref{fig:zz6}. A list of all Wilf-equivalence classes for size $5$ and $6$ zig-zag classes and their generating functions can be found at the Github repository \url{https://github.com/jaypantone/bipartite-pops}.

\begin{figure}\centering
	\input{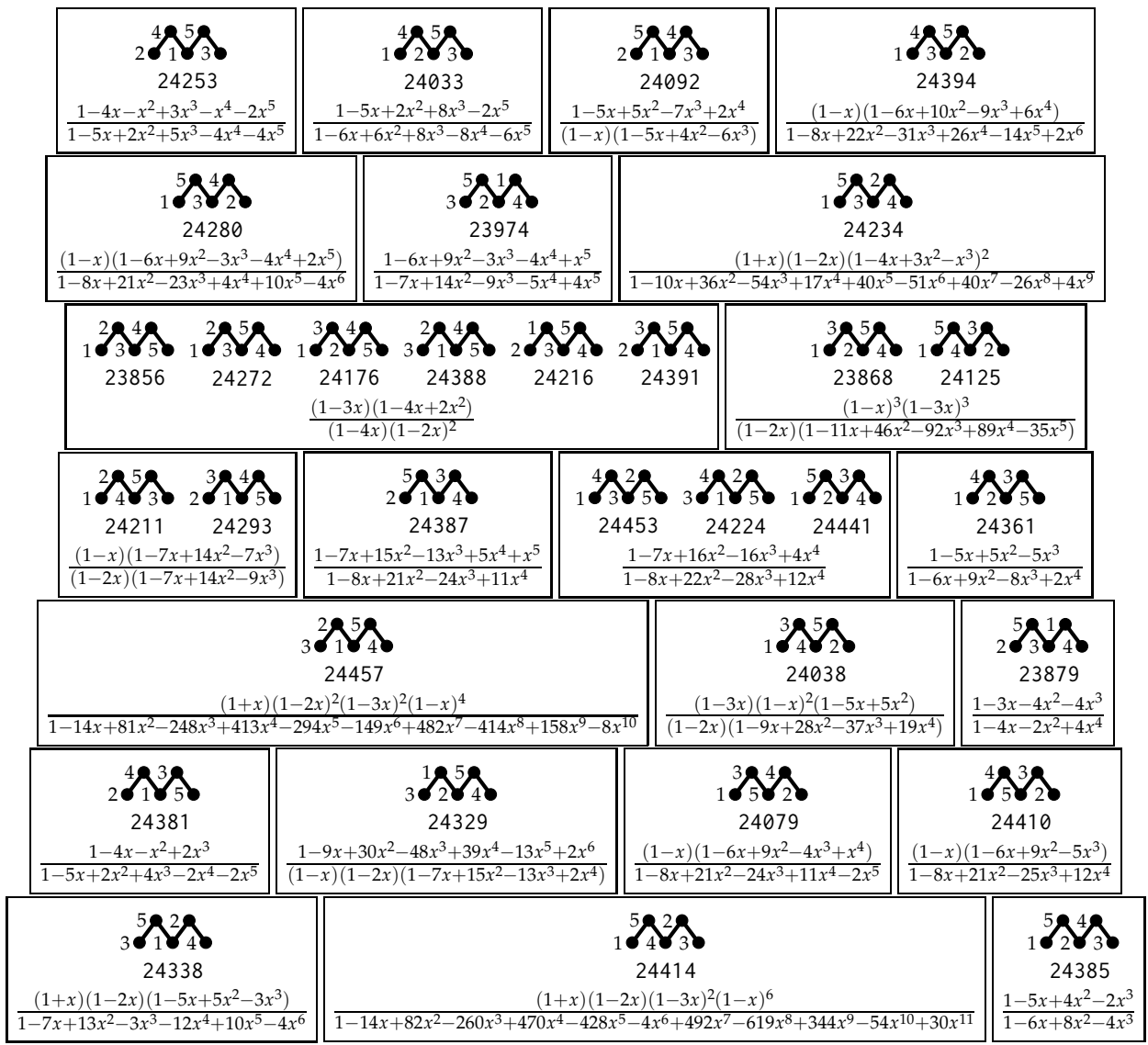}
	\caption{The Wilf-equivalence classes of size $5$ zig-zag classes. Only one class from each symmetry class is shown. Below each poset is the PermPal ID of the class.}
	\label{fig:M-classes}
\end{figure}

\begin{figure}
	\input{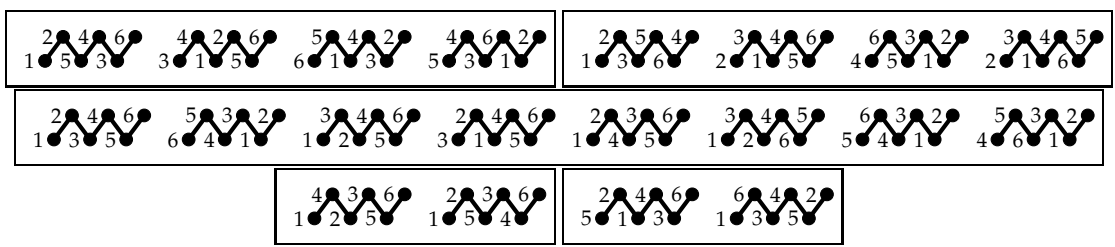}
	\caption{The nontrivial Wilf-equivalence classes of size $6$ zig-zag classes. All other Wilf-equivalence classes consist of only a single symmetry class.}
	\label{fig:zz6}
\end{figure}

In contrast to the situation for general POP classes described in Section~\ref{subsection:symmetry}, we conjecture that the two operations described by Gao and Kitaev fully describe all symmetry classes of zig-zag classes.
\begin{conjecture}
	If two zig-zag classes are symmetrically equivalent, then one is formed by  complementing the labels of the other or by reversing all of the relations of the other, or both.
\end{conjecture}
If this conjecture is true, then the number of symmetry classes of zig-zag classes is given by sequence A262480 in the OEIS.

\section{Using Tilescope to enumerate size $5$ POP classes}
\label{section:tilescope}

Combinatorial Exploration introduced by Albert, Bean, Claesson, Nadeau, Pantone, and Ulfarsson~\cite{combinatorial-exploration} is a domain-agnostic algorithmic framework for discovering combinatorial specifications for sets of combinatorial objects. A combinatorial specification is a way of describing the structure of a set of objects, and from a specification, one can often derive many useful results, including a generating function of the set, polynomial counting algorithms, random sampling routines, and more. This framework is experimental because you are not always guaranteed to find a specification in finite time; however, it is verifiably correct once such a result is found.

The Tilescope software package~\cite{tilings-bibtex} is an implementation of Combinatorial Exploration for the field of permutation patterns. Although Tilescope is guaranteed to find a specification for every permutation class with a regular insertion encoding (which was used to produce Figures~\ref{fig:M-classes} and \ref{fig:zz6}), there is no such guarantee in general. In \cite{combinatorial-exploration} Tilescope was used to enumerate every permutation class whose basis contained only size $4$ patterns except for $\Av(1324)$, which has no known polynomial-time enumeration formula. As a result, one can consider all size $4$ POP classes to have been enumerated except $\Av(1324)$. All of the enumerative results of \cite{combinatorial-exploration} are cataloged on the Permutation Pattern Avoidance Library (PermPAL) \cite{permpal-bibtex}, which can be found at \url{https://permpal.com}.

We applied TileScope to size 5 POP classes. In total there are 4231 size 5 POP classes, but only 1068 after considering symmetric equivalence. Of these 1068 classes, 223 have a regular insertion encoding, so TileScope was able to find a specification for these, and we were able to compute the rational generating functions for these POPs. For the remaining 845, TileScope found a specification for 590 of them. We computed the generating function for 223 of these $590$ using automatic kernel method style techniques. For the other $367$ of the $590$, we could not derive a generating function either because the system of equations has more than one catalytic variable or because our automatic methods exceeded available memory or time. Even in these cases, the specification still produces a polynomial-time algorithm that can be used to generate terms of the counting sequence.

\begin{table}
	\centering
	\begin{tblr}{
		colspec = {*{13}{Q[c,m]}}, 
		hline{1, 3} = {1pt, black}, 
				column{1} = {bg=gray!15, font=\bfseries}, 
			}
		Min. Poly. Order  & 1   & 2  & 3  & 4  & 5 & 6  & 7 & 8 & 9 & 10 & 11 & 12 \\
		\# of size 5 POPs & 223 & 93 & 42 & 32 & 8 & 22 & 7 & 6 & 7 & 5  & 0  & 1
	\end{tblr}
	\caption{The number of POPs of size $5$ with a given degree of minimal polynomial.}
	\label{tab:degree-of-minimal-polynomial}
\end{table}

Table~\ref{tab:degree-of-minimal-polynomial} shows the order of the minimal polynomials of the $446$ size $5$ POP classes for which Tilescope finds a specification and we are able to solve the resulting system of equations. For example, there are five size $5$ POP classes for which we found an algebraic generating function with minimal polynomial of order $10$. One example is $\Av(13542, \allowbreak 14523, \allowbreak 14532, \allowbreak 15324, \allowbreak 15423, \allowbreak 15432, \allowbreak 24513, \allowbreak 25314, \allowbreak 25413)$ for which the minimal polynomial is
\begin{align*}
	 & x^{7} F \left(x
	\right)^{10}-x^{6} \left(3 x +2\right) F \left(x
	\right)^{9}+x^{5} \left(3 x^{2}+6 x -1\right) F \left(x
	\right)^{8}-x^{4} \left(x -1\right) \left(x^{2}+7 x +4\right) F \left(x
	\right)^{7}                                                              \\
	 & \qquad +x^{3} \left(2 x +1\right) \left(x^{2}-3 x -2\right) F \left(x
	\right)^{6}+x^{2} \left(3 x^{3}+7 x^{2}-3 x +1\right) F \left(x
	\right)^{5}                                                              \\
	 & \qquad -x^{2} \left(3 x^{2}-4 x -3\right) F \left(x
	\right)^{4} -x \left(x +2\right) \left(x +1\right) F \left(x
	\right)^{3}+2 x^{2} F \left(x
	\right)^{2}+\left(x +1\right) F \! \left(x \right)-1.
\end{align*}

It is notable that all of the size $5$ POP classes for which we found a rational generating function are bipartite and thus have a regular insertion encoding guaranteeing this rationality. This leads us to conjecture that all POPs with a rational generating function have a regular insertion encoding.

\begin{conjecture}
	A POP class has a rational generating function if and only if it is bipartite.
\end{conjecture}

The proof of Theorem~\ref{theorem:rie-condition} shows that every permutation class that does not have a regular insertion encoding has a subclass $\Av(\sigma)$ where $\sigma$ is a permutation of size $3$. This subclass has an algebraic generating function. This, however, does not prove that such a class cannot have a rational generating function; consider for example the result of Albert, Brignall, and Vatter~\cite{albert:large-infinite-antichains} that says every proper permutation class is the subclass of a permutation class with a rational generating function.

\section{Resolving some conjectures of Gao and Kitaev}
\label{section:conjectures}

In Gao and Kitaev~\cite{gao:partially-ordered-patterns}, Table~5 lists six different POP classes of size $5$ that do not have a regular insertion encoding and for each one makes a conjecture that their counting sequence is the same as an existing OEIS sequence. We have reproduced their Table 5 below, here labeled Table~\ref{table:their-table-5}. For ease of reference, we will refer to their conjectures as Conjecture \gk{1} through Conjecture \gk{6} according to their row in the table. Through the use of Tilescope to enumerate many size 5 POP classes, we are able to refute Conjecture \gk{1}, affirm Conjectures \gk{2}, \gk{3}, \gk{4}, and \gk{5}, and provide more numerical evidence for the correctness of Conjecture \gk{6}.

\begin{table}
	\centering
	\begin{tblr}{
		colspec = {Q[c,m] *{6}{Q[c,m]}}, 
		hline{1, 3} = {1pt, black}, 
				column{1} = {bg=gray!15, font=\bfseries}, 
			}
		Poset           &
		\gkoneposet{}   &
		\gktwoposet{}   &
		\gkthreeposet{} &
		\gkfourposet{}  &
		\gkfiveposet{}  &
		\gksixposet{}     \\
		OEIS            &
		A216879         &
		A054872         &
		A118376         &
		A212198         &
		A228907         &
		A224295
	\end{tblr}
	\caption{A reproduction, in part, of Table 5 from~\cite{gao:partially-ordered-patterns} listing the six conjectures addressed here.}
	\label{table:their-table-5}
\end{table}

\subsection{Conjecture \gk{1}}

Conjecture \gk{1} involves the class $\GG_1 = \Av\Biggl(\gkoneposet{}\Biggr) = \Av(31425, \allowbreak 31524, \allowbreak 32415, \allowbreak 32514, \allowbreak 41235, \allowbreak 41325, \allowbreak 42135, \allowbreak 42315, \allowbreak 43125, \allowbreak 43215)$ and the OEIS sequence A216879 defined as the solution to an equation involving Jacobi theta functions. The conjecture that the counting sequence for $\GG_1$ is (a shift of) A216879 is incorrect. One can use a software library like Permuta \cite{permuta-bibtex} or PermLab \cite{albert:permlab} to quickly check that the number of permutations of size $10$ in $\GG_1$ is $443,592$, while the corresponding term in A216879 is $443,594$.

Furthermore, Tilescope was able to find a combinatorial specification for $\GG_1$.\pplink{24008} This allows us to determine that the generating function for $\GG_1$ is algebraic, satisfying the minimal polynomial
\[
	(4x-1)F(x)^{4}-(16x - 6)F(x)^{3}+(x^{2}+24x-13)F(x)^{2}-(16x-12) F(x)+4x-4 = 0
\]
The counting sequence for this class has been added to the OEIS as entry A366706.

\subsection{Conjecture \gk{2}}

Conjecture \gk{2} involves the class $\GG_2 = \Av\Biggl(\gktwoposet{}\Biggr) = \Av(41235, 41325, \allowbreak 42135, \allowbreak 42315, \allowbreak 43125, \allowbreak 43215)$ and the OEIS sequence A054872 defined as the number of permutations of size $n$ in the class $\Av(12345, 13245, 21345, 23145, 31245, 32145)$. Note that these two classes are not symmetrically equivalent.

Tilescope found a combinatorial specification for $\GG_2$\pplink{23934} that leads to a univariate system of 136 equations that can be solved to derive the generating function
\[
	1 + 2x - 2x^2 - x\sqrt{1-8x+4x^2}
\]
which matches the generating function given in the OEIS sequence, thus confirming Conjecture \gk{2}. Tilescope has also found a combinatorial specification for the class described by the OEIS sequence.\pplink{23656}.

The PermPAL database actually contains six additional permutation classes with this same generating function, plus two more whose generating function we could not find but whose enumeration matches up to at least size $100$, leading us to conjecture that they are also Wilf-equivalent.

Table~\ref{table:A054872} lists these ten classes, all of which turn out to be POPs.\footnote{It is not necessarily the case that if a class $\CC$ is a POP class, and $\DD$ is Wilf-equivalent to $\CC$, then $\DD$ or some symmetry of $\DD$ will be a POP. For example, $\Av(132, 231)$ is a POP, and it is Wilf-equivalent to the symmetry class $\{\Av(132, 213), \Av(231, 312)\}$, neither of which are POPs.}  On PermPAL, all classes are stored using the symmetry with the lexicographically smallest basis. That symmetry may not be a POP, so here we choose a symmetry of each that is a POP and in most cases gives the same shape. The PermPAL entry for each class can be accessed at the url ``http://permpal.com/perms/[PermPAL ID]'' where ``[PermPAL ID]'' is replaced by the ID given in the table. Note that the first row is the class in the definition of the OEIS sequence and the second to last row is the POP that Conjecture \gk{2} is concerned with.

\begin{table}
	\begin{center}
		\begin{tblr}{
			colspec= {Q[c, h]Q[c]Q[c]Q[c]},
			hlines = {black, 1pt},
			rowsep = 0pt,
			stretch = 0,
			row{1} = {font=\bfseries},
			}
			\SetRow{gray!15, rowsep=8pt}
			Poset                       & Permutation Class                               & PermPAL ID & GF known? \\
			\broom{1}{2}{3}{4}{5}       & $\Av(12345, 13245, 21345, 23145, 31245, 32145)$ & 23656      & yes       \\
			\broom{1}{2}{3}{5}{4}       & $\Av(12354, 13254, 21354, 23154, 31254, 32154)$ & 23660      & yes       \\
			\broom{1}{2}{4}{3}{5}       & $\Av(12435, 13425, 21435, 23415, 31425, 32415)$ & 23663      & yes       \\
			\widediamond{1}{2}{4}{5}{3} & $\Av(12534, 12543, 13524, 13542, 14523, 14532)$ & 23949      & no        \\
			\broom{1}{2}{5}{3}{4}       & $\Av(12453, 13452, 21453, 23451, 31452, 32451)$ & 23667      & yes       \\
			\broom{1}{2}{4}{5}{3}       & $\Av(12534, 13524, 21534, 23514, 31524, 32514)$ & 23669      & no        \\
			\broom{1}{2}{5}{4}{3}       & $\Av(12543, 13542, 21543, 23541, 31542, 32541)$ & 23671      & yes       \\
			\broom{1}{3}{4}{2}{5}       & $\Av(14235, 14325, 24135, 24315, 34125, 34215)$ & 23673      & yes       \\
			\broom{2}{3}{4}{1}{5}       & $\Av(41235, 41325, 42135, 42315, 43125, 43215)$ & 23934      & yes       \\
			\broom{1}{3}{5}{2}{4}       & $\Av(14253, 14352, 24153, 24351, 34152, 34251)$ & 23678      & yes
		\end{tblr}
	\end{center}
	\caption{The ten permutation classes known or conjectured to be counted by OEIS A054872.}
	\label{table:A054872}
\end{table}

\subsection{Conjecture \gk{3}}

Conjecture \gk{3} involves the class $\GG_3 = \Av\Biggl(\gkthreeposet{}\Biggr) = \Av(51423, \allowbreak 51432, \allowbreak 52413, \allowbreak 52431, \allowbreak 53412, \allowbreak 53421, \allowbreak 54312, \allowbreak 54321)$ and the OEIS sequence A118376 defined as the ``number of trees of weight $n$ where nodes have positive integer weights and the sum of the weights of the children of a node is equal to the weight of the node.'' Tilescope finds a specification for $\GG_3$\pplink{23703} that leads to a bivariate system of 26 equations that can be solved with kernel method techniques to find the generating function
\[
	\frac{5 - 4x -\sqrt{1 - 8x + 8x^2}}{4 - 4x}
\]
which matches (a shift of) the generating function given for A118376.

In fact, there are sixteen classes in the PermPAL database that appear to have this same enumeration, shown in Table~\ref{table:A118376}. For twelve of these, we can solve the resulting system of equations to find the generating function shown above. The remaining four have either bivariate systems of equations that are too big for us to solve or trivariate systems of equations for which no solving method is known; for these four we confirm using their specifications that their enumerations match A118376 up to size $100$.

\begin{table}
	\begin{center}
		\begin{tblr}{
			colspec= {Q[c, h]X[c, m]Q[c]Q[c]},
			hlines = {black, 1pt},
			rowsep = 0pt,
			stretch = 0,
			row{1} = {font=\bfseries},
			width = 0.9\linewidth
			}
			\SetRow{gray!15, rowsep=8pt}
			Poset                   & Permutation Class & PermPAL ID & GF known? \\
			\treepop{1}{2}{4}{3}{5} &
			$\Av(12345, 12435, 13425, 21345,$
				\newline \hspace*{14pt} $21435, 23415, 31425, 32415)$
			                        & 23703             & yes
			\\
			\fishpop{1}{2}{3}{5}{4} &
			$\Av(12354, 12453, 13254, 13452,
				$ \newline\hspace*{14pt} $14253, 14352, 15243, 15342)$
			                        & 24143             & yes
			\\
			\treepop{1}{2}{3}{5}{4} &
			$\Av(12354, 12453, 13254, 21354,$
				\newline \hspace*{14pt} $21453, 23154, 31254, 32154)$
			                        & 23704             & yes
			\\
			\treepop{1}{2}{5}{3}{4} &
			$\Av(12354, 12453, 13452, 21354,$
				\newline \hspace*{14pt} $21453, 23451, 31452, 32451)$
			                        & 23705             & yes
			\\
			\treepop{1}{3}{4}{2}{5} &
			$\Av(13245, 14235, 14325, 23145,$
				\newline \hspace*{14pt} $24135, 24315, 34125, 34215)$
			                        & 23948             & yes
			\\
			\treepop{2}{3}{4}{1}{5} &
			$\Av(31245, 32145, 41235, 41325,$
				\newline \hspace*{14pt} $42135, 42315, 43125, 43215)$
			                        & 24088             & yes
			\\
			\fishpop{1}{2}{4}{5}{3} &
			$\Av(12534, 12543, 13524, 13542,
				$ \newline\hspace*{14pt} $14523, 14532, 15423, 15432)$
			                        & 24094             & yes
			\\
			\fishpop{3}{2}{4}{5}{1} &
			$\Av(45123, 45132, 52134, 52143,
				$ \newline\hspace*{14pt} $53124, 53142, 54123, 54132)$
			                        & 23711             & yes
			\\
			\treepop{1}{2}{4}{5}{3} &
			$\Av(12534, 12543, 13524, 21534,$
				\newline \hspace*{14pt} $21543, 23514, 31524, 32514)$
			                        & 23712             & no
			\\
			\treepop{1}{5}{2}{4}{3} &
			$\Av(13542, 14523, 14532, 23541,$
				\newline \hspace*{14pt} $24513, 24531, 34512, 34521)$
			                        & 23714             & no
			\\
			\treepop{1}{3}{5}{2}{4} &
			$\Av(13254, 14253, 14352, 23154,$
				\newline \hspace*{14pt} $24153, 24351, 34152, 34251)$
			                        & 23717             & no
			\\
			\fishpop{2}{3}{4}{5}{1} &
			$\Av(41523, 41532, 51234, 51243,
				$ \newline\hspace*{14pt} $51324, 51342, 51423, 51432)$
			                        & 23721             & yes
			\\
			\fishpop{2}{1}{4}{5}{3} &
			$\Av(21534, 21543, 31524, 31542,
				$ \newline\hspace*{14pt} $41523, 41532, 51423, 51432)$
			                        & 23723             & no
			\\
			\treepop{1}{5}{3}{2}{4} &
			$\Av(13452, 14253, 14352, 23451,$
				\newline \hspace*{14pt} $24153, 24351, 34152, 34251)$
			                        & 23724             & yes
			\\
			\treepop{2}{3}{5}{1}{4} &
			$\Av(31254, 32154, 41253, 41352,$
				\newline \hspace*{14pt} $42153, 42351, 43152, 43251)$
			                        & 24002             & yes
			\\
			\fishpop{2}{1}{3}{5}{4} &
			$\Av(21354, 21453, 31254, 31452,
				$ \newline\hspace*{14pt} $41253, 41352, 51243, 51342)$
			                        & 24376             & yes
		\end{tblr}
	\end{center}
	\caption{The sixteen permutation classes known or conjectured to be counted by OEIS A118376.}
	\label{table:A118376}
\end{table}

\subsection{Conjecture \gk{4}}
\label{subsection:gk4}

Conjecture \gk{4} involves the class $\GG_4 = \Av\Biggl(\gkfourposet{}\Biggr) = \Av(45123, \allowbreak 45213, \allowbreak 54123, \allowbreak 54213)$ and the OEIS sequence A212198 defined as the ``leading diagonal of triangle in A211321'', which itself is defined as enumerating ``marked mesh patterns of type $R_n^{(2,0,2,0)}$'' as defined by Kitaev and Remmel~\cite{kitaev:quadrant-marked-mesh}. The set of permutations that avoid marked mesh patterns of this type form the permutation class $\Av(45312, 45321, 54312, 54321)$. Martinez and Savage~\cite{martinez:inv-seqs-2} proved that this leading diagonal sequence is also the counting sequence for inversion sequences avoiding the patterns $201$ and $210$. Chen and Lin~\cite{chen:length-5-patts-and-inv-201-210} and Pantone~\cite{pantone:inv-201-210} have enumerated these, finding their generating function to be
\[
	\frac{2-x-x\sqrt{1-8 x}}{2 \left(1 - 2x + 2 x^{2}\right)}.
\]
Chen and Lin~\cite{chen:length-5-patts-and-inv-201-210} then found bijections that linked four permutation classes ($\GG_4$, the class defined by A212198, and two more classes avoiding four patterns of size five) with inversion sequences avoiding the patterns $201$ and $210$. They also conjecture that nine further classes defined by avoiding four patterns of size five have the same counting sequence.

We have determined by an exhaustive search that in addition to the thirteen classes that Chen and Lin either prove or conjecture are enumerated by A212198, there are at most ten more with the same counting sequence whose basis contains only size $5$ patterns. We have enumerated the nine classes in Chen and Lin's conjecture, as well as six of the ten new classes. For an additional one of the ten, we find a combinatorial specification that allows us to compute the first 100 terms of the counting sequence (which match A212198), but not compute the generating function. We summarize below the state of the classes now known or conjectured to be counted by A212198. For each class whose enumeration we claim, a specification can be found on PermPAL by entering the basis into the search bar. Note that we have listed the symmetry with the lexicographically minimal basis for each class.

\begin{multicols}{2}
\textbf{Enumerated in~\cite{chen:length-5-patts-and-inv-201-210}:}\\
\hspace*{0.2in} $\Av(12345, 12354, 21345, 21354)$\\
\hspace*{0.2in} $\Av(12453, 12543, 21453, 21543)$\\
\hspace*{0.2in} $\Av(13524, 13542, 31524, 31542)$\\
\hspace*{0.2in} $\Av(14352, 15342, 24351, 25341)$
	
\textbf{Conjectured in~\cite{chen:length-5-patts-and-inv-201-210}, proved here:}\\
\hspace*{0.2in} $\Av(13425, 13452, 31425, 31452)$\\
\hspace*{0.2in} $\Av(13452, 13542, 23451, 23541)$\\
\hspace*{0.2in} $\Av(13452, 13542, 31452, 31542)$\\
\hspace*{0.2in} $\Av(12435, 12453, 21435, 21453)$\\
\hspace*{0.2in} $\Av(13542, 14532, 23541, 24531)$\\
\hspace*{0.2in} $\Av(12354, 12453, 21354, 21453)$\\
\hspace*{0.2in} $\Av(13524, 15324, 23514, 25314)$\\
\hspace*{0.2in} $\Av(13425, 13524, 31425, 31524)$\\
\hspace*{0.2in} $\Av(13542, 15342, 23541, 25341)$

\textbf{Proved here:}\\
\hspace*{0.2in} $\Av(12354, 12435, 21354, 21435)$\\
\hspace*{0.2in} $\Av(12453, 12534, 21453, 21534)$\\
\hspace*{0.2in} $\Av(13524, 15324, 31524, 51324)$\\
\hspace*{0.2in} $\Av(13542, 14352, 23541, 24351)$\\
\hspace*{0.2in} $\Av(13524, 15324, 31524, 35124)$\\
\hspace*{0.2in} $\Av(13524, 14253, 23514, 24153)$

\textbf{Specification, but no generating function:}\\
\hspace*{0.2in} $\Av(12453, 14253, 21453, 41253)$

\textbf{Conjectured here:}\\
\hspace*{0.2in} $\Av(14325, 14352, 41325, 41352)$\\
\hspace*{0.2in} $\Av(13524, 23514, 25314, 31524)$\\
\hspace*{0.2in} $\Av(13542, 14352, 31542, 41352)$
\end{multicols}

A symmetry of the class $\GG_4$ is among those enumerated in~\cite{chen:length-5-patts-and-inv-201-210}. A specification can be found on PermPAL\pplink{23641} that leads to a univariate system of 22 equations that can be solved to confirm the generating function.

In total, eleven of the classes listed above are POPS: the four enumerated in~\cite{chen:length-5-patts-and-inv-201-210}, the first six conjectured by~\cite{chen:length-5-patts-and-inv-201-210}, and the first that we conjecture here but cannot find a generating function for. The ten that we can enumerate are listed in Table~\ref{table:A212198}.

\begin{table}
	\begin{center}
		\begin{tblr}{
			colspec= {Q[c, h]X[c, m]Q[c]Q[c]},
			hlines = {black, 1pt},
			rowsep = 0pt,
			stretch = 0,
			row{1} = {font=\bfseries},
			width = 0.9\linewidth
			}
			\SetRow{gray!15, rowsep=8pt}
			Poset                    & Permutation Class & PermPAL ID & GF known? \\
			\xpop{1}{2}{3}{4}{5}     &
			$\Av(12345, 12354, 21345, 21354)$
			                         & 23634             & yes                    \\
			\housepop{1}{2}{3}{5}{4} &
			$\Av(12354, 12453, 21354, 21453)$
			                         & 24403             & yes                    \\
			\xpop{1}{2}{4}{3}{5}     &
			$\Av(12435, 12534, 21435, 21534)$
			                         & 24019             & yes                    \\
			\xpop{1}{2}{5}{3}{4}     &
			$\Av(12453, 12543, 21453, 21543)$
			                         & 23641             & yes                    \\
			\xpop{1}{3}{4}{2}{5}     &
			$\Av(14235, 15234, 24135, 25134)$
			                         & 23642             & yes                    \\
			\xpop{1}{5}{2}{3}{4}     &
			$\Av(13452, 13542, 23451, 23541)$
			                         & 23643             & yes                    \\
			\housepop{1}{3}{4}{5}{2} &
			$\Av(15234, 15243, 25134, 25143)$
			                         & 23644             & yes                    \\
			\xpop{1}{3}{5}{2}{4}     &
			$\Av(14253, 15243, 24153, 25143)$
			                         & 23645             & yes                    \\
			\housepop{1}{5}{2}{4}{3} &
			$\Av(13542, 14532, 23541, 24531)$
			                         & 24212             & yes                    \\
			\xpop{1}{5}{3}{2}{4}     &
			$\Av(13542, 14532, 23541, 24531)$
			                         & 24172             & yes
		\end{tblr}
	\end{center}
	\caption{The ten POP classes known to be counted by OEIS A212198.}
	\label{table:A212198}
\end{table}

\subsection{Conjecture \gk{5}}
Conjecture \gk{5} involves the class $\GG_5 = \Av\Biggl(\gkfiveposet{}\Biggr) = \Av(45123, \allowbreak 45132, \allowbreak 45213, \allowbreak 54123, \allowbreak 54132, \allowbreak 54213)$ and (a shift of) the OEIS sequence A228907 defined as the series expansion of the generating function $F(x)$ satisfying the equation
\[
	F(x) = 1 + \sum_{n \geq 0} \frac{1 - F(x)^{2n}}{1 - F(x)}x^n,
\]
which is algebraic with minimal polynomial
\[
	x(x-1)F(x)^3 - x(x-1)F(x)^2 - (2x-1)F(x) - 1 = 0.
\]

Tilescope finds a specification for $\GG_5$\pplink{2364} that leads to a system of 20 equations with one catalytic variable that can be solved to find the shifted generating function $1 + xF(x)$ where $F(x)$ is the generating function from A228907. This is the only permutation class in the PermPAL database with this counting sequence.

\subsection{Conjecture \gk{6}}
Conjecture \gk{6} involves the class $\GG_6 = \Av\Biggl(\gksixposet{}\Biggr) = \Av(45123, \allowbreak 45213)$ and the OEIS sequence A224295 which is defined as the counting sequence of $\Av(12345, 12354)$. These two classes are not symmetries, but $\Av(12345, 12354)$ is also a POP class.

We have not been able to enumerate $\GG_6$, and so we cannot confirm this conjecture. However, Tilescope is able to find a specification for the class $\Av(12345, 12354)$ in the definition of the sequence. It leads to a system of 23 equations with \emph{two} catalytic variables, preventing us from solving it. We are able to generate the first $790$ terms of the counting sequence of this class. PermPal has an additional six classes, each avoiding two patterns of size five, whose counting sequences match that of $\Av(12345, 12354)$ up to at least the first $100$ terms.

As for the class $\GG_6$, we have computed the first $50$ terms in the counting sequence through different means, and they match the sequence A224295 up to this point, providing further evidence for this still-open conjecture.

\bibliographystyle{alpha}
\bibliography{paper.bib}

\end{document}